\documentclass[a4, 12pt]{amsart}
\usepackage[mathscr]{eucal}
\usepackage{amssymb}
\usepackage{latexsym}
\usepackage{amsthm}
\usepackage{color}
\theoremstyle{plain}
\newtheorem{theorem}{Theorem}[section]

\newtheorem{remark}{Remark}[section]
\newtheorem{lemma}{Lemma}[section]
\newtheorem{proposition}{Proposition}[section]

\makeatletter
\@addtoreset{equation}{section}

\title[Classification of complete $3$-dimensional self-shrinkers]
{A classification of complete $3$-dimensional self-shrinkers in  Euclidean space $\mathbb R^{4}$}
\author [Q. -M. Cheng, Z. Li and G. Wei]{Qing-Ming Cheng,  Zhi Li  and Guoxin Wei}

\address{Qing-Ming Cheng \\  \newline \indent Department of Applied Mathematics, Faculty of Science,
\newline \indent Fukuoka University, Fukuoka  814-0180, Japan.}
\email{cheng@fukuoka-u.ac.jp}

\address{Zhi Li \\  \newline \indent College of Mathematics and Information Science, Henan Normal University,
\newline \indent 453007, Xinxiang, Henan, China.}
\email{lizhihnsd@126.com}

\address{Guoxin Wei \\ \newline \indent School of Mathematical Sciences, South China Normal University,
\newline \indent 510631, Guangzhou,  China.}
\email{weiguoxin@tsinghua.org.cn}

\begin{document}
\maketitle

\begin{abstract}
In this paper, we completely  classify $3$-dimensional complete self-shrinkers with constant norm $S$ of
the second fundamental form  and  constant $f_{3}$ in Euclidean space $\mathbb R^{4}$, where $h_{ij}$
are components of the second fundamental form, $S=\sum_{i,j}h^{2}_{ij}$ and $f_{3}=\sum_{i,j,k}h_{ij}h_{jk}h_{ki}$.
\end{abstract}

\footnotetext{The first author was partially supported by JSPS Grant-in-Aid for Scientific Research:
No.16H03937 and No. 22K03303 and the  fund of Fukuoka University: No. 225001.
The second author was partially supported by  China Postdoctoral Science Foundation Grant No. 2022M711074.
The third author was partly supported by grant No.  12171164 of NSFC,
GDUPS (2018), Guangdong
Natural Science Foundation Grant No.2023A1515010510.}

\section{introduction}
\vskip2mm
\noindent

It is well-known that  mean curvature flow has been used to model various things  in material sciences and physics
such as cell, bubble growth and so on. Hence, study on the
mean curvature flow is a very important subject in  the differential geometry.
One of the most important problems in the  mean curvature flow is to understand
the possible singularities that the flow goes through.  In order to describe
singularities of the mean curvature flow,  self-shrinkers  of the mean curvature  flow  play a very key  role. A hypersurface
$X: M^{n}\to \mathbb{R}^{n+1}$ of Euclidean space $\mathbb{R}^{n+1}$  is called {\it a self-shrinker of the mean curvature flow} if it satisfies
\begin{equation}\label{1.1-1}
 H+ \langle X, e_{n+1}\rangle =0,
\end{equation}
where  $e_{n+1}$  and $H$ denote the normal vector and the mean curvature  of  $X: M^{n}\to \mathbb{R}^{n+1}$, respectively.

\noindent
For classifications of complete self-shrinkers,
many nice works have been done.  Abresch and Langer \cite{AL}  classified
closed self-shrinker curves completely. These curves are so-called Abresch-Langer curves.
In \cite{H2}, Huisken proved that sphere $S^n(\sqrt n)$ is the only
$n$-dimensional  compact self-shrinkers
in $\mathbb{R}^{n+1}$ with non-negative mean curvature. Furthermore, in \cite{H3} and \cite{CM}, Huisken and
Colding and Minicozzi  completely classified complete self-shrinkers with
non-negative mean curvature and polynomial volume growth.
According to the results  of Halldorsson \cite{H}, Ding and Xin \cite{DX1}, Cheng and Zhou \cite{CZ}, one knows
that $\gamma \times \mathbb R^{n-1}$ is a complete self-shrinker without polynomial volume growth in $\mathbb R^{n+1}$, where
$\gamma$ is a complete self-shrinking curve of Halldorsson  \cite{H}. Hence, the  condition of polynomial volume growth in \cite{CM}
is essential. On the other hand, Colding and Minicozzi  \cite{CM}   (cf.  Andrews, Li and Wei  \cite{ALW}, Arezzo and Sun \cite{AS}  for  higher
co-dimensions)  gave a classification of  n-dimensional $\mathcal F$-stable complete  self-shrinker with polynomial volume growth in  $\mathbb{R}^{n+1}$.  Furthermore, Cao and Li \cite{CL} (cf. Le and Sesum \cite{LS}) gave
a classification  for complete self-shrinkers of the mean curvature flow with polynomial volume growth and $S \leq 1$,
where $S$ denotes the squared norm of the second fundamental form.
By making use of the generalized maximum principle,  Cheng and Peng \cite{CP} studied complete self-shrinkers of the mean curvature flow without the assumption of  polynomial volume growth.
For study on the rigidity of complete self-shrinkers, many works have been done
(cf. \cite{CHW}, \cite{CP}, \cite{cw2}, \cite{lw}, \cite{LXX}, \cite{LW1}, \cite{LW2}, \cite{W},  and so on).

\noindent
For complete self-shrinkers with constant squared norm of the second fundamental form, the following conjecture is well-known and very
important (cf. \cite{cw1}, \cite{CLW}).

\vskip2mm
\noindent
{\bf Conjecture}. {\it  An $n$-dimensional complete self-shrinker  $X: M\rightarrow \mathbb{R}^{n+1}$
 with constant  squared norm of the second fundamental form
is isometric to one of  $S^{n}(\sqrt{n})$, $\mathbb{R}^{n}$, $S^k (\sqrt k)\times \mathbb{R}^{n-k}$, $1\leq k\leq n-1$.}
\vskip2mm

In \cite{DX2}, by estimating the first eigenvalue of the Dirichlet  eigenvalue problem,
Ding and Xin  proved
that a $2$-dimensional complete self-shrinker  $X: M^{2}\rightarrow \mathbb{R}^{3}$  with polynomial volume growth
and with constant  squared norm $S$ of the second fundamental form
is isometric to one of  $S^{2}(\sqrt{2})$, $\mathbb{R}^{2}$,
$S^{1} (1)\times \mathbb{R}$.
Recently, Cheng and Ogata \cite{CO}  have solved  the above conjecture for $n=2$ affirmatively,  that is, they have proved the following:

\vskip3mm
\noindent
{\bf Theorem CO.}
{\it A $2$-dimensional complete self-shrinker  $X: M\rightarrow \mathbb{R}^{3}$  with constant  squared norm of the second fundamental form
is isometric to one of $S^{2}(\sqrt{2})$, $\mathbb{R}^{2}$, $S^1 (1)\times \mathbb{R}$.}

\vskip2mm

Cheng and Peng \cite{CP}  proved that an $n$-dimensional complete self-shrinker  $X: M\rightarrow \mathbb{R}^{n+1}$
 with constant  squared norm of the second fundamental form
is isometric to one of  $S^k (\sqrt k)\times \mathbb{R}^{n-k}$, $1\leq k\leq n-1$, and $S^{n}(\sqrt{n})$ if $\inf H^2>0$.
Hence, in order to solve the above conjecture, ones only need to prove $H\equiv 0$ if $\inf H^2=0$.
\vskip2mm
\noindent
For the  dimension $3$, by studying  the infimum of $H^2$, Cheng, Li and Wei \cite{CLW} have proved the above conjecture is true
under the assumption that  $f_4=\sum_{i,j,k,l}h_{ij}h_{jk}h_{kl}h_{li}$ is constant. In this paper, by considering the supremum and  the infimum of $H$, we prove the above
conjecture is true if  $f_{3}$ is constant. In fact, we prove the following:

\begin{theorem}\label{theorem 1.1}
 Let $X: M^{3}\to \mathbb{R}^{4}$ be a
$3$-dimensional complete  self-shrinker in $\mathbb R^{4}$.
If the squared norm $S$ of the second fundamental form and $f_{3}$ are constant, then $X: M^{3}\to \mathbb{R}^{4}$ is
isometric to one of
\begin{enumerate}
\item $S^{3}(\sqrt{3})$,
\item $\mathbb {R}^{3}$,
\item $S^{1}(1)\times \mathbb{R}^{2}$,
\item $S^{2}(\sqrt{2})\times \mathbb{R}^{1}$.
\end{enumerate}
In particular, $S$ must be $0$ and $1$; $f_{3}$ must be $0$, $1$, $\frac{\sqrt{2}}{2}$ and $\frac{\sqrt{3}}{3}$, where $h_{ij}$ are the components of the second fundamental form, $S=\sum\limits_{i,j}h_{ij}^2$ and  $f_{3}=\sum\limits_{i,j,k}h_{ij}h_{jk}h_{ki}$.
\end{theorem}

\begin{remark}
Recently, Cheng, Wei and Yano \cite{CWY} proved that,  for an $n$-dimensional complete self-shrinker  $X: M\rightarrow \mathbb{R}^{n+1}$
  in $\mathbb{R}^{n+1}$,
if the squared norm $S$ of the second fundamental form  and $f_3$ are  constant and $S$ satisfies
$$
S\leq 1.83379,
$$
then $M$ is isometric to one of  $S^{n}(\sqrt{n})$, $\mathbb{R}^{n}$,  $S^{k} (\sqrt k)\times \mathbb{R}^{n-k}$,  $1\leq k\leq n-1$.
\end{remark}

\vskip5mm
\section {Preliminaries}
\vskip2mm

\noindent

For an $n$-dimensional hypersurface $X: M^{n} \rightarrow\mathbb{R}^{n+1}$  of $(n+1)$-dimensional Euclidean space $\mathbb{R}^{n+1}$,
we choose a local orthonormal frame field
$\{e_{A}\}_{A=1}^{n+1}$ in $\mathbb{R}^{n+1}$ with dual co-frame field
$\{\omega_{A}\}_{A=1}^{n+1}$, such that, restricted to $M^{n}$,
$e_{1},\cdots, e_{n}$  are tangent to $M^{n}$. We make  use of  the following conventions on the ranges of indices,
$$
 1\leq i,j,k,l\leq n.
$$
 $\sum_{i}$ means taking  summation from $1$ to $n$ for $i$.
Then we have
\begin{equation*}
dX=\sum_i\limits \omega_{i} e_{i}
\end{equation*}
and the Levi-Civita connection $\omega_{ij}$ of the hypersurface satisfies
\begin{equation*}
de_{i}=\sum_{j}\limits \omega_{ij}e_{j}+\omega_{i n+1}e_{n+1},
\end{equation*}
\begin{equation*}
de_{n+1}=\omega_{n+1 i}e_{i}, \ \ \omega_{n+1i}=-\omega_{in+1}.
\end{equation*}

\noindent By  restricting  these forms to $M$,  we get
\begin{equation}\label{2.1-1}
\omega_{n+1}=0.
\end{equation}

\noindent Thus, we obtain, from Cartan lemma,
\begin{equation}\label{2.1-2}
\omega_{in+1}=\sum_{j} h_{ij}\omega_{j},\quad
h_{ij}=h_{ji}.
\end{equation}

$$
h=\sum_{i,j}h_{ij}\omega_i\otimes\omega_j, \ \ H= \sum_i\limits h_{ii}
$$
are called  second fundamental form and mean curvature  of $X: M\rightarrow\mathbb{R}^{n+1}$, respectively.
Let $S=\sum_{i,j}\limits (h_{ij})^2$ be  the squared norm
of the second fundamental form  of $X: M\rightarrow\mathbb{R}^{n+1}$.
The Gauss equations of hypersurface are given by
\begin{equation}\label{2.1-3}
R_{ijkl}=h_{ik}h_{jl}-h_{il}h_{jk}.
\end{equation}

\noindent
Defining
covariant derivative of $h_{ij}$ by
\begin{equation}\label{2.1-4}
\sum_{k}h_{ijk}\omega_k=dh_{ij}+\sum_kh_{ik}\omega_{kj}
+\sum_k h_{kj}\omega_{ki},
\end{equation}
we obtain the Codazzi equations
\begin{equation}\label{2.1-5}
h_{ijk}=h_{ikj}.
\end{equation}
By defining
\begin{equation}\label{2.1-6}
\sum_lh_{ijkl}\omega_l=dh_{ijk}+\sum_lh_{ljk}\omega_{li}
+\sum_lh_{ilk}\omega_{lj}+\sum_l h_{ijl}\omega_{lk}
\end{equation}
and
\begin{equation}\label{2.1-8}
\begin{aligned}
\sum_mh_{ijklm}\omega_m&=dh_{ijkl}+\sum_mh_{mjkl}\omega_{mi}
+\sum_mh_{imkl}\omega_{mj}\\
&+\sum_mh_{ijml}\omega_{mk}
+\sum_mh_{ijkm}\omega_{ml}
\end{aligned}
\end{equation}
we have the following Ricci identities
\begin{equation}\label{2.1-7}
h_{ijkl}-h_{ijlk}=\sum_m
h_{mj}R_{mikl}+\sum_m h_{im}R_{mjkl}.
\end{equation}
\begin{equation}\label{2.1-9}
\begin{aligned}
h_{ijklq}-h_{ijkql}&=\sum_{m} h_{mjk}R_{milq}
+ \sum_{m}h_{imk}R_{mjlq}+ \sum_{m}h_{ijm}R_{mklq}.
\end{aligned}
\end{equation}
For a smooth function $f$, we define
\begin{equation*}
df:=\sum_i f_{,i}\omega_i, \ \ \sum_j f_{,ij}\omega_j:=df_{,i}+\sum_j
f_{,j}\omega_{ji}.
\end{equation*}
Therefore, we know that  the norm of gradient $\nabla f$ and Laplacian of $f$ are given by
\begin{equation*}
|\nabla f|^2=\sum_{i }(f_{,i})^2,\ \ \Delta f =\sum_i f_{,ii}.
\end{equation*}

\noindent
We define  functions  $f_{3}$  and $f_4$ as follows:
$$f_{3}=\sum_{i,j,k}h_{ij}h_{jk}h_{ki}, \ \ \ f_{4}=\sum_{i,j,k,l}h_{ij}h_{jk}h_{kl}h_{li}.
$$

\noindent
As in \cite{CM}, we define the $\mathcal{L}$-operator  by
\begin{equation*}
\mathcal{L}f=\Delta f-\langle X,\nabla f\rangle.
\end{equation*}

\noindent  A hypersurface  $X: M\rightarrow\mathbb{R}^{n+1}$
is called {\it a self-shrinker of mean curvature flow} if
$$H+\langle X, e_{n+1}\rangle =0.
$$
By a simple calculation, we have the following basic formulas.

\begin{equation}\label{2.1-11}
\aligned
H_{,i}
=&\sum_{k}h_{ik}\langle X, e_{k}\rangle, \\
H_{,ij}
=&\sum_{k}h_{ijk}\langle X, e_{k}\rangle+h_{ij}-H\sum_{k}h_{ik}h_{kj}.
\endaligned
\end{equation}

\noindent
Using the above formulas and the Ricci identities, we can get the following Lemma (cf. \cite{CW}, \cite{CLW}):

\begin{lemma}\label{lemma 2.1}
Let $X:M^n\rightarrow \mathbb{R}^{n+1}$ be an $n$-dimensional   self-shrinker in $\mathbb R^{n+1}$. We have
\begin{equation}\label{2.1-12}
\mathcal{L}H=H(1-S),
\end{equation}
\begin{equation}\label{2.1-13}
\aligned
\frac{1}{2}\mathcal{L}
H^{2}=|\nabla H|^{2}+H^{2}(1-S),
\endaligned
\end{equation}
\begin{equation}\label{2.1-14}
\frac{1}{2}\mathcal{L}S
=\sum_{i,j,k}h_{ijk}^{2}+S(1-S),
\end{equation}
\begin{equation}\label{2.1-15}
\frac{1}{3}\mathcal{L}f_{3}
=2\sum_{i,j,k,l}h_{ijl}h_{jkl}h_{ki}+(1-S)f_{3}.
\end{equation}
\end{lemma}

\noindent
\begin{lemma}\label{lemma 2.2}
Let $X:M^{n}\rightarrow \mathbb{R}^{n+1}$ be an $n$-dimensional  self-shrinker in $\mathbb R^{n+1}$. If $S$ is constant,  we have
\begin{equation}\label{2.1-16}
\aligned
\sum_{i,j,k,l}h_{ijkl}^{2}&=(S-2)\sum_{i,j,k}h_{ijk}^{2}\\
&-6\sum_{i,j,k,l,p}h_{ijk}h_{il}h_{jp}h_{klp}
+3\sum_{i,j,k,l,p}h_{ijk}h_{ijl}h_{kp}h_{lp}.
\endaligned
\end{equation}
\end{lemma}
\begin{proof}
Since $S$ is constant, we know from the lemma 2.1,
$$
\sum_{i,j,k}h_{ijk}^{2}=S(S-1).
$$
By making use of the Ricci identities  \eqref{2.1-7},  \eqref{2.1-9} and a direct calculation, we
have
\begin{equation*}
\aligned
&\frac{1}{2}\mathcal{L}\sum_{i,j,k}h_{ijk}^{2}=\sum_{i,j,k,l}h_{ijkl}^{2}-(S-2)\sum_{i,j,k}(h_{ijk})^{2}\\
&+6\sum_{i,j,k,l,p}h_{ijk}h_{il}h_{jp}h_{klp}
-3\sum_{i,j,k,l,p}h_{ijk}h_{ijl}h_{kp}h_{lp}.
\endaligned
\end{equation*}
Thus we  obtain \eqref{2.1-16} since $\sum_{i,j,k}h_{ijk}^{2}$ is constant.
\end{proof}

\noindent In order to prove our theorem, we need the following generalized maximum principle  which is proved by Cheng and Peng \cite{CP}.

\vskip2mm
\noindent
\begin{lemma}\label{lemma 2.3}
Let $X : M^{n}\to \mathbb{R}^{n+p}$ be a complete self-shrinker with Ricci curvature bounded from below.
Let $f$ be any $C^{2}$-function bounded from above on this self-shrinker.
Then, there exists a sequence of points $\{p_{t}\}\in M^{n}$, such that
\begin{equation*}
\lim_{t\rightarrow\infty} f(X(p_{t}))=\sup f,\quad
\lim_{t\rightarrow\infty} |\nabla f|(X(p_{t}))=0,\quad
\limsup_{t\rightarrow\infty}\mathcal{L}f(X(p_{t}))\leq 0.
\end{equation*}
\end{lemma}

 \vskip10mm
\section{Proof of Theorem 1.1}

\vskip2mm

From \eqref{2.1-14}, we know that $S=0$ or $S=1$ or $S>1$.  If $S=0$, we know that $X: M^{3}\to \mathbb{R}^{4}$ is $\mathbb{R}^{3}$. Next, we assume that $S\geq1$. We will prove the following

\begin{proposition}\label{theorem 3.1}
For a $3$-dimensional complete self-shrinker $X:M^{3}\rightarrow \mathbb{R}^{4}$ with nonzero constant squared norm $S$ of the second fundamental form,  if $f_{3}$ is constant,
we have that either $S=1$ or $\sup H=\frac{3f_{3}}{2S}$.
\end{proposition}

\begin{proof}
\noindent
Since $S$ is constant, if $S\leq 1$, according to the results of Cheng and Peng \cite{CP}, we have $S=1$.
Thus, we only need to consider $S>1$. We choose $e_{1}$, $e_{2}$ and $e_{3}$,  at each point $p\in M^{3}$,  such that
$$
h_{ij}=\lambda_i\delta_{ij}.
$$
\noindent From the definitions of $S$ and $H$, we obtain
$$
H^{2}\leq 3S.
$$
\noindent Since $S$ is constant, from the Gauss
equations, we know that the Ricci curvature of $X:M^{3} \to \mathbb{R}^{4}$
is bounded from below.
We can apply the generalized maximum principle for $\mathcal{L}$-operator
to $H$. Thus, there exists a sequence $\{p_{t}\}$ in $M^{3}$ such that
\begin{equation}\label{3.1-1}
\lim_{t\rightarrow\infty} H(X(p_{t}))=\sup H,\quad
\lim_{t\rightarrow\infty} |\nabla H|(X(p_{t}))=0,\quad
\lim_{t\rightarrow\infty}\sup\mathcal{L}H(X(p_{t}))\leq 0.
\end{equation}

\noindent
Since $S$ and $f_{3}$ are constant, by \eqref{2.1-14} and \eqref{2.1-16}, we know that
$\{h_{ij}(p_{t})\}$,  $\{h_{ijk}(p_{t})\}$ and $\{h_{ijkl}(p_{t})\}$ are bounded sequences for $ i, j, k, l = 1,2,3$.
Hence, we can assume that they convergence if necessary , taking a subsequence.
\begin{equation*}
\begin{aligned}
&\lim_{t\rightarrow\infty}H(p_{t})=\bar H, \ \ \lim_{t\rightarrow\infty}h_{ij}(p_{t})=\bar h_{ij}=\bar \lambda_i\delta_{ij}, \\
&\lim_{t\rightarrow\infty}h_{ijk}(p_{t})=\bar h_{ijk}, \ \  \lim_{t\rightarrow\infty}h_{ijkl}(p_{t})=\bar h_{ijkl}, \ \ i, j, k,l=1, 2, 3.
\end{aligned}
\end{equation*}

\noindent
By taking exterior derivative of $H$, from \eqref{3.1-1} and by taking limits,
we know
\begin{equation}\label{3.1-2}
\bar H_{,i}=\bar h_{11i}+\bar h_{22i}+\bar h_{33i}=0, \ \ i=1, 2, 3.
\end{equation}

\noindent
According to the definition of  the self-shrinker, we have
$$H_{,i}=\sum_{k}h_{ik}\langle X, e_{k}\rangle, \ \  \ \ i=1, 2, 3,$$
$$H_{,ij}=\sum_{k}h_{ijk}\langle X, e_{k}\rangle+h_{ij}-H\sum_{k}h_{ik}h_{kj}, \ \ i,j=1, 2, 3.
$$
Thus, we get
\begin{equation}\label{3.1-3}
\bar H_{,i}=\bar h_{11i}+\bar h_{22i}+\bar h_{33i}=\bar \lambda_{i}\lim_{t\rightarrow\infty} \langle X, e_{i} \rangle(p_{t}),\ \ \ \ i=1, 2, 3,
\end{equation}
and
\begin{equation}\label{3.1-4}
\begin{cases}
\begin{aligned}
&\bar H_{,ii}=\sum_{k}\bar h_{iik}\lim_{t\rightarrow\infty} \langle X, e_{k} \rangle(p_{t})+\bar \lambda_{i}-\bar H\bar \lambda^{2}_{i},\ \ i=1,2,3, \\
&\bar H_{,ij}=\sum_{k}\bar h_{ijk}\lim_{t\rightarrow\infty} \langle X, e_{k} \rangle(p_{t}),\ \ i\neq j, \ \ i,j=1,2,3.
\end{aligned}
\end{cases}
\end{equation}

\noindent Since $S$ is constant, we obtain
$$
\sum_{i,j}h_{ij}h_{ijk}=0, \ \  \ k=1, 2, 3,
$$
$$
\sum_{i,j}h_{ij}h_{ijkl}+\sum_{i,j}h_{ijk}h_{ijl}=0, \ \  \ k,l=1, 2, 3.
$$
Under the processing  by taking limits, we have
\begin{equation}\label{3.1-5}
\bar\lambda_{1}\bar h_{11k}+\bar\lambda_{2}\bar h_{22k}+\bar\lambda_{3}\bar h_{33k}=0, \ \  \ k=1, 2, 3,
\end{equation}
and
\begin{equation}\label{3.1-6}
\begin{cases}
\begin{aligned}
&\sum_{i}\bar \lambda_{i}\bar h_{iikk}=-\sum_{i,j}\bar h^{2}_{ijk}, \ \  \ k=1, 2, 3, \\
&\sum_{i}\bar \lambda_{i}\bar h_{iikl}=-\sum_{i,j}\bar h_{ijk}\bar h_{ijl}, \ \ k\neq l, \ \  \ k,l=1, 2, 3.
\end{aligned}
\end{cases}
\end{equation}

\noindent It follows from Ricci identities \eqref{2.1-7} that
\begin{equation*}
\bar h_{ijkl}-\bar h_{ijlk}=\bar\lambda_{i}\bar\lambda_{j}\bar\lambda_{k}\delta_{il}\delta_{jk}-\bar\lambda_{i}\bar\lambda_{j}\bar\lambda_{l}\delta_{ik}\delta_{jl}
+\bar\lambda_{i}\bar\lambda_{j}\bar\lambda_{k}\delta_{ik}\delta_{jl}-\bar\lambda_{i}\bar\lambda_{j}\bar\lambda_{l}\delta_{il}\delta_{jk},
\end{equation*}
that is,
\begin{equation}\label{3.1-7}
\begin{cases}
\begin{aligned}
& \bar h_{1212}-\bar h_{1221}=\bar \lambda_{1}\bar \lambda_{2}(\bar \lambda_{1}-\bar \lambda_{2}),\ \ \bar h_{1313}-\bar h_{1331}=\bar \lambda_{1}\bar \lambda_{3}(\bar \lambda_{1}-\bar \lambda_{3}),\\
&\bar h_{2323}-\bar h_{2332}=\bar \lambda_{2}\bar \lambda_{3}(\bar \lambda_{2}-\bar \lambda_{3}), \ \ \bar h_{iikl}-\bar h_{iilk}=0, \ \ \ i,k,l=1, 2, 3.
\end{aligned}
\end{cases}
\end{equation}

\noindent Since $f_{3}$ is constant,  by \eqref{2.1-11}, we know that

\begin{equation}\label{3.1-8}
\bar \lambda^{2}_{1}\bar h_{11k}+\bar \lambda^{2}_{2}\bar h_{22k}+\bar \lambda^{2}_{3}\bar h_{33k}=0, \ \  k=1, 2, 3,
\end{equation}
and
\begin{equation}\label{3.1-9}
\begin{cases}
\begin{aligned}
&\sum_{i}\bar \lambda^{2}_{i}\bar h_{iikk}=-2\sum_{i,j}\bar \lambda_{i}\bar h^{2}_{ijk}, \ \  \ k=1, 2, 3, \\
&\sum_{i}\bar \lambda^{2}_{i}\bar h_{iikl}=-2\sum_{i,j}\bar \lambda_{i}\bar h_{ijk}\bar h_{ijl}, \ \ k\neq l, \ \  \ k,l=1, 2, 3.
\end{aligned}
\end{cases}
\end{equation}
From the above equations  \eqref{3.1-2}   to \eqref{3.1-9}, we can prove the following claim:
\vskip1mm
\noindent
{\bf Claim}. If $S>1$, we have
$$
 \bar H=\frac{3f_{3}}{2S}
$$
holds.\newline
In fact, if   $\bar \lambda_1=\bar \lambda_2 =\bar \lambda_3$,  we have
\begin{equation}\label{3.1-10}
\bar H^{2}=3S, \ \ \bar H=3\bar \lambda_{1}.
\end{equation}
From \eqref{3.1-6} and \eqref{3.1-9}, we have
\begin{equation*}
\begin{cases}
\begin{aligned}
\bar \lambda_{1}\sum_{i}h_{ii11}
=&-(\bar h^{2}_{111}+\bar h^{2}_{221}+\bar h^{2}_{331})-2(\bar h^{2}_{112}+\bar h^{2}_{113}+\bar h^{2}_{123}),\\
\bar \lambda_{1}\sum_{i}h_{ii22}
=&-(\bar h^{2}_{112}+\bar h^{2}_{222}+\bar h^{2}_{332})-2(\bar h^{2}_{221}+\bar h^{2}_{223}+\bar h^{2}_{123}),\\
\bar \lambda_{1}\sum_{i}h_{ii33}
=&-(\bar h^{2}_{113}+\bar h^{2}_{223}+\bar h^{2}_{333})-2(\bar h^{2}_{331}+\bar h^{2}_{332}+\bar h^{2}_{123}),
\end{aligned}
\end{cases}
\end{equation*}
and
\begin{equation*}
\begin{cases}
\begin{aligned}
&\bar \lambda^{2}_{1}\sum_{i}h_{ii11}
=-2\bar \lambda_{1}(\bar h^{2}_{111}+\bar h^{2}_{221}+\bar h^{2}_{331})-4\bar \lambda_{1}(\bar h^{2}_{112}+\bar h^{2}_{113}+\bar h^{2}_{123}),\\

&\bar \lambda^{2}_{1}\sum_{i}h_{ii22}
=-2\bar \lambda_{1}(\bar h^{2}_{112}+\bar h^{2}_{222}+\bar h^{2}_{332})-4\bar \lambda_{1}(\bar h^{2}_{221}+\bar h^{2}_{223}+\bar h^{2}_{123}),\\

&\bar \lambda^{2}_{1}\sum_{i}h_{ii33}
=-2\bar \lambda_{1}(\bar h^{2}_{113}+\bar h^{2}_{223}+\bar h^{2}_{333})-4\bar \lambda_{1}(\bar h^{2}_{331}+\bar h^{2}_{332}+\bar h^{2}_{123}).
\end{aligned}
\end{cases}
\end{equation*}
Thus, we infer
\begin{equation}\label{3.1-11}
\bar h_{ijk}=0, \ \ i,j,k=1, 2, 3.
\end{equation}
According to the lemma 2.1, we know $S=1$. It is impossible because of $S>1$.

\vskip2mm
\noindent
If   two of  $\bar \lambda_1$, $\bar \lambda_2$ and $\bar \lambda_3$  are equal,
without loss of generality, we assume that $\bar \lambda_{1}\neq \bar \lambda_{2}=\bar \lambda_{3}$.
Then we get from\eqref{3.1-2} and \eqref{3.1-5}

\begin{equation}\label{3.1-12}
\bar h_{11k}=0, \ \ \bar h_{22k}+\bar h_{33k}=0,\ \ \ k=1, 2, 3.
\end{equation}
\newline
If  $\bar \lambda_{1}=0$, then $\bar \lambda_{2}=\bar \lambda_{3} \neq 0$ because of $S> 1$.
By making use of  equations \eqref{3.1-6}, \eqref{3.1-9} and \eqref{3.1-12}, we know that

\begin{equation*}
\begin{cases}
\begin{aligned}
&\bar \lambda_{2}(\bar h_{2211}+\bar h_{3311})=-2(\bar h^{2}_{221}+\bar h^{2}_{123}),\\
&\bar \lambda_{2}(\bar h_{2222}+\bar h_{3322})=-2(\bar h^{2}_{222}+\bar h^{2}_{223})-2(\bar h^{2}_{221}+\bar h^{2}_{123}),
\end{aligned}
\end{cases}
\end{equation*}
and
\begin{equation*}
\begin{cases}
\begin{aligned}
\bar \lambda^{2}_{2}(\bar h_{2211}+\bar h_{3311})=&-4\bar\lambda_{2}(\bar h^{2}_{221}+\bar h^{2}_{123}),\\

\bar \lambda^{2}_{2}(\bar h_{2222}+\bar h_{3322})=&-4\bar\lambda_{2}(\bar h^{2}_{222}+\bar h^{2}_{223})-2\bar\lambda_{2}(\bar h^{2}_{221}+\bar h^{2}_{123}).
\end{aligned}
\end{cases}
\end{equation*}
Hence, we get
\begin{equation*}
\bar h_{221}=\bar h_{123}=0, \ \ \bar h_{222}=\bar h_{223}=0.
\end{equation*}
Namely, we obtain
\begin{equation}\label{3.1-13}
\bar h_{ijk}=0, \ \ i,j,k=1, 2, 3.
\end{equation}
We have $S=1$,  which contradicts  to $S>1$.
\newline
If $\bar \lambda_{1}\neq0, \ \ \bar \lambda_{2}=\bar \lambda_{3}=0$,
combining it with \eqref{3.1-6}, \eqref{3.1-9} and \eqref{3.1-12}, we have

\begin{equation*}
\begin{cases}
\begin{aligned}
&\bar \lambda_{1}\bar h_{1111}
=-2(\bar h^{2}_{221}+\bar h^{2}_{123}),\\
&\bar \lambda_{1}\bar h_{1122}
=-2(\bar h^{2}_{222}+\bar h^{2}_{223})-2(\bar h^{2}_{221}+\bar h^{2}_{123}),
\end{aligned}
\end{cases}
\end{equation*}
and
\begin{equation*}
\begin{cases}
\begin{aligned}
&\bar \lambda^{2}_{1}\bar h_{1111}=0,\\

&\bar \lambda^{2}_{1}\bar h_{1122}=-2\bar \lambda_{1}(\bar h^{2}_{221}+\bar h^{2}_{123}).
\end{aligned}
\end{cases}
\end{equation*}
We infer
\begin{equation*}
\bar h_{221}=\bar h_{123}=0, \ \ \bar h_{222}=\bar h_{223}=0.
\end{equation*}
Thus, we conclude
\begin{equation}\label{3.1-14}
\bar h_{ijk}=0, \ \ i,j,k=1, 2, 3.
\end{equation}
Hence, we infer $S=1$.  It is also impossible because of $S>1$.
\newline
If $\bar \lambda_{1}\neq 0, \ \ \bar \lambda_{2}=\bar \lambda_{3} \neq 0$,
combining \eqref{3.1-2} and \eqref{3.1-3}, we obtain

\begin{equation}\label{3.1-15}
\lim_{t\rightarrow\infty} \langle X, e_{k} \rangle(p_{t})=0,\ \ \ k=1, 2, 3.
\end{equation}

\noindent
It follows from \eqref{3.1-4}, \eqref{3.1-6}, \eqref{3.1-9}, \eqref{3.1-12} and \eqref{3.1-15} that

\begin{equation}\label{3.1-16}
\begin{cases}
\begin{aligned}
&\bar h_{1111}+\bar h_{2211}+\bar h_{3311}=\bar \lambda_{1}-\bar H\bar \lambda^{2}_{1},\\
&\bar h_{1122}+\bar h_{2222}+\bar h_{3322}=\bar \lambda_{2}-\bar H\bar \lambda^{2}_{2},\\
&\bar h_{1133}+\bar h_{2233}+\bar h_{3333}=\bar \lambda_{2}-\bar H\bar \lambda^{2}_{2},\\
&\bar h_{1112}+\bar h_{2212}+\bar h_{3312}= 0,\\
&\bar h_{1113}+\bar h_{2213}+\bar h_{3313}= 0,\\
&\bar h_{1123}+\bar h_{2223}+\bar h_{3323}= 0,
\end{aligned}
\end{cases}
\end{equation}

\begin{equation}\label{3.1-17}
\begin{cases}
\begin{aligned}
\bar \lambda_{1}\bar h_{1111}+\bar \lambda_{2}(\bar h_{2211}+\bar h_{3311})&=-2(\bar h^{2}_{221}+\bar h^{2}_{123}),\\
\bar \lambda_{1}\bar h_{1122}+\bar \lambda_{2}(\bar h_{2222}+\bar h_{3322})&=-2(\bar h^{2}_{222}+\bar h^{2}_{223})
-2(\bar h^{2}_{221}+\bar h^{2}_{123}),\\
\bar \lambda_{1}\bar h_{1133}+\bar \lambda_{2}(\bar h_{2233}+\bar h_{3333})&=-2(\bar h^{2}_{222}+\bar h^{2}_{223})
-2(\bar h^{2}_{221}+\bar h^{2}_{123}),\\
\bar \lambda_{1}\bar h_{1112}+\bar \lambda_{2}(\bar h_{2212}+\bar h_{3312})&=-2(\bar h_{221}\bar h_{222}+\bar h_{223}\bar h_{123}),\\
\bar \lambda_{1}\bar h_{1113}+\bar \lambda_{2}(\bar h_{2213}+\bar h_{3313})&=-2(\bar h_{221}\bar h_{223}-\bar h_{222}\bar h_{123}),\\
\bar \lambda_{1}\bar h_{1123}+\bar \lambda_{2}(\bar h_{2223}+\bar h_{3323})&=0,
\end{aligned}
\end{cases}
\end{equation}
and
\begin{equation}\label{3.1-18}
\begin{cases}
\begin{aligned}
\bar \lambda^{2}_{1}\bar h_{1111}+\bar \lambda^{2}_{2}(\bar h_{2211}+\bar h_{3311})=&-4\bar\lambda_{2}(\bar h^{2}_{221}+\bar h^{2}_{123}),\\

\bar \lambda^{2}_{1}\bar h_{1122}+\bar \lambda^{2}_{2}(\bar h_{2222}+\bar h_{3322})=&-4\bar\lambda_{2}(\bar h^{2}_{222}+\bar h^{2}_{223})\\
&-2(\bar\lambda_{1}+\bar\lambda_{2})(\bar h^{2}_{221}+\bar h^{2}_{123}),\\

\bar \lambda^{2}_{1}\bar h_{1133}+\bar \lambda^{2}_{2}(\bar h_{2233}+\bar h_{3333})=&-4\bar\lambda_{2}(\bar h^{2}_{222}+\bar h^{2}_{223})\\
&-2(\bar\lambda_{1}+\bar\lambda_{2})(\bar h^{2}_{221}+\bar h^{2}_{123}),\\

\bar \lambda^{2}_{1}\bar h_{1112}+\bar \lambda^{2}_{2}(\bar h_{2212}+\bar h_{3312})=&-4\bar\lambda_{2}(\bar h_{221}\bar h_{222}+\bar h_{223}\bar h_{123}),\\

\bar \lambda^{2}_{1}\bar h_{1113}+\bar \lambda^{2}_{2}(\bar h_{2213}+\bar h_{3313})=&-4\bar\lambda_{2}(\bar h_{221}\bar h_{223}-\bar h_{222}\bar h_{123}),\\

\bar \lambda^{2}_{1}\bar h_{1123}+\bar \lambda^{2}_{2}(\bar h_{2223}+\bar h_{3323})=&0.
\end{aligned}
\end{cases}
\end{equation}

\noindent Furthermore, \eqref{3.1-16} yields
\begin{equation}\label{3.1-19}
\bar h_{2212}+\bar h_{3312}= -\bar h_{1112},\ \
\bar h_{2213}+\bar h_{3313}= -\bar h_{1113}.
\end{equation}
Inserting \eqref{3.1-19} into \eqref{3.1-17} and \eqref{3.1-18}, we derive  to
\begin{equation}\label{3.1-20}
\bar h_{221}\bar h_{222}+\bar h_{223}\bar h_{123}=0, \ \ \bar h_{221}\bar h_{223}-\bar h_{222}\bar h_{123}=0.
\end{equation}

\noindent It follows from \eqref{3.1-17} and \eqref{3.1-18} that

\begin{equation}\label{3.1-21}
\begin{cases}
\begin{aligned}
&\bar \lambda_{1}(\bar \lambda_{2}-\bar \lambda_{1})\bar h_{1111}= 2\bar\lambda_{2}(\bar h^{2}_{221}+\bar h^{2}_{123}),\\

&\bar \lambda_{1}(\bar \lambda_{2}-\bar \lambda_{1})\bar h_{1122}= 2\bar\lambda_{2}(\bar h^{2}_{222}+\bar h^{2}_{223})+2\bar\lambda_{1}(\bar h^{2}_{221}+\bar h^{2}_{123}),\\

&\bar \lambda_{1}(\bar \lambda_{2}-\bar \lambda_{1})\bar h_{1133}= 2\bar\lambda_{2}(\bar h^{2}_{222}+\bar h^{2}_{223})+2\bar\lambda_{1}(\bar h^{2}_{221}+\bar h^{2}_{123}).
\end{aligned}
\end{cases}
\end{equation}
From \eqref{3.1-20}, we have $\bar h^{2}_{221}+\bar h^{2}_{123} = 0$,  or  $ \bar h^{2}_{222}+\bar h^{2}_{223} = 0$.
\vskip1mm
If both $\bar h^{2}_{221}+\bar h^{2}_{123}=0$  and $ \bar h^{2}_{222}+\bar h^{2}_{223} =0$,
we have
\begin{equation*}
\bar h_{ijk}=0, \ \ i,j,k=1, 2, 3.
\end{equation*}
From the lemma 2.1, $S=1$.   It is a contradiction.
\vskip1mm
If  $\bar h^{2}_{221}+\bar h^{2}_{123}=0$ and  $\bar h^{2}_{222}+\bar h^{2}_{223} \neq 0$,
by making use of \eqref{3.1-21}, we have
\begin{equation*}
\begin{cases}
\begin{aligned}
&\bar \lambda_{1}(\bar \lambda_{2}-\bar \lambda_{1})\bar h_{1111}= 0,\\

&\bar \lambda_{1}(\bar \lambda_{2}-\bar \lambda_{1})\bar h_{1122}= 2\bar\lambda_{2}(\bar h^{2}_{222}+\bar h^{2}_{223}),\\

&\bar \lambda_{1}(\bar \lambda_{2}-\bar \lambda_{1})\bar h_{1133}= 2\bar\lambda_{2}(\bar h^{2}_{222}+\bar h^{2}_{223}).
\end{aligned}
\end{cases}
\end{equation*}
Thus,
\begin{equation}\label{3.1-22}
\bar h_{1111}= 0, \ \ \bar h_{1122}=\bar h_{1133}=\frac{2\bar\lambda_{2}}{\bar \lambda_{1}(\bar \lambda_{2}-\bar \lambda_{1})}(\bar h^{2}_{222}+\bar h^{2}_{223}).
\end{equation}

\noindent
Noting that $\bar h_{1111}=0$ and by  \eqref{3.1-17}, we know
\begin{equation*}
\bar h_{2211}+\bar h_{3311}=0, \ \ \bar H_{,11}=0.
\end{equation*}
Hence, by the first equation of \eqref{3.1-16}, we have
\begin{equation}\label{3.1-23}
\bar \lambda_{1}\bar H=1.
\end{equation}
Combining \eqref{3.1-16}, \eqref{3.1-17} with \eqref{3.1-23}, we know
\begin{equation*}
\begin{aligned}
-2(\bar h^{2}_{222}+\bar h^{2}_{223})&=\bar \lambda_{1}\bar h_{1122}+\bar \lambda_{2}\bigg(\bar \lambda_{2}-\bar H\bar \lambda^{2}_{2}-\bar h_{1122}\bigg) \\
&=\bar \lambda_{2}\bigg(\bar \lambda_{2}-\bar H\bar \lambda^{2}_{2}\bigg)+(\bar \lambda_{1}-\bar \lambda_{2})\bar h_{1122} \\
&=\bar \lambda_{2}\bigg(\bar \lambda_{2}-\bar H\bar \lambda^{2}_{2}\bigg)+(\bar \lambda_{1}-\bar \lambda_{2}) \cdot \frac{2\bar\lambda_{2}}{\bar \lambda_{1}(\bar \lambda_{2}-\bar \lambda_{1})}(\bar h^{2}_{222}+\bar h^{2}_{223}) \\
&=\bar \lambda_{2}\bigg(\bar \lambda_{2}-\bar H\bar \lambda^{2}_{2}\bigg)-\frac{2\bar\lambda_{2}}{\bar \lambda_{1}}(\bar h^{2}_{222}+\bar h^{2}_{223}).
\end{aligned}
\end{equation*}
From \eqref{3.1-23}, we obtain
\begin{equation*}
\bar h^{2}_{222}+\bar h^{2}_{223}=-\frac{\bar \lambda^{2}_{2}}{2}.
\end{equation*}
It is a contradiction.

\vskip1mm
If  $\bar h^{2}_{221}+\bar h^{2}_{123} \neq 0$ and $ \bar h^{2}_{222}+\bar h^{2}_{223}=0$,
acccording to \eqref{2.1-14} and \eqref{2.1-15} in Lemma \ref{lemma 2.2}, we have
\begin{equation}\label{3.1-24}
\sum_{i,j,k}h_{ijk}^{2}=S(S-1), \ \
2\sum_{i,j,k,l}h_{ijl}h_{jkl}h_{ki}=f_{3}(S-1).
\end{equation}

\noindent Since $\bar h^{2}_{222}+\bar h^{2}_{223}=0$, it follows from \eqref{3.1-12} that
\begin{equation}\label{3.1-25}
\sum_{i,j,k}\bar h_{ijk}^{2}=6(\bar h^{2}_{221}+\bar h^{2}_{123}), \ \ 2\sum_{i,j,k,l}\bar h_{ijl}\bar h_{jkl}\bar h_{ki}
=4\bar H(\bar h^{2}_{221}+\bar h^{2}_{123}).
\end{equation}
Therefore, by \eqref{3.1-24} and \eqref{3.1-25}, we obtain
\begin{equation*}
6(\bar h^{2}_{221}+\bar h^{2}_{123})=S(S-1), \ \ 4\bar H(\bar h^{2}_{221}+\bar h^{2}_{123})=(S-1)f_{3},
\end{equation*}
that is,
\begin{equation*}
(2\bar H S-3f_{3})(S-1)=0.
\end{equation*}
Therefore, we conclude
\begin{equation*}
2\bar H S-3f_{3}=0, \ \ \bar H=\frac{3f_{3}}{2S},
\end{equation*}
because of $S>1$.

\vskip2mm
\noindent
If  $\bar \lambda_1$, $\bar \lambda_{2}$ and $\bar \lambda_3$ are distinct,
by using of \eqref{3.1-2}, \eqref{3.1-5} and \eqref{3.1-8}, we have that
\begin{equation}\label{3.1-26}
\bar h_{11k}=\bar h_{22k}=\bar h_{33k}=0, \ \ k=1, 2, 3,
\end{equation}

\noindent then
\begin{equation}\label{3.1-27}
\sum_{i,j,k}\bar h_{ijk}^{2}=6\bar h^{2}_{123}, \ \ 2\sum_{i,j,k,l}\bar h_{ijl}\bar h_{jkl}\bar h_{ki}
=4\bar H\bar h^{2}_{123}.
\end{equation}

\noindent It follows from the lemma \ref{lemma 2.1}
\begin{equation*}
6\bar h^{2}_{123}=S(S-1), \ \ 4\bar H\bar h^{2}_{123}=(S-1)f_{3},
\end{equation*}
namely,
\begin{equation*}
 (2\bar H S-3f_{3})(S-1)=0.
\end{equation*}
Hence, we get
\begin{equation*}
2\bar H S-3f_{3}=0, \ \ \bar H =\frac{3f_{3}}{2S}.
\end{equation*}
Thus, we complete  the  proof  of  the proposition 3.1.
\end{proof}

\noindent
By applying the generalized maximum principle for $\mathcal{L}$-operator to the function $-H$ and using the same proof as that of proving  the proposition \ref{theorem 3.1}, we can obtain the following
\begin{proposition}\label{theorem 3.2}
For a $3$-dimensional complete self-shrinker $X:M^{3}\rightarrow \mathbb{R}^{4}$ with nonzero constant squared norm $S$ of the second fundamental form. If $f_{3}$ is constant,
then we have that either $S=1$ or $\inf H=\frac{3f_{3}}{2S}$.
\end{proposition}

\vskip3mm
\noindent
{\it Proof of Theorem \ref{theorem 1.1}}.
If $S=0$, we know that $X: M^{3}\to \mathbb{R}^{4}$ is the  plane $\mathbb{R}^{3}$. If $S\neq0$,
from the proposition \ref{theorem 3.1} and the proposition \ref{theorem 3.2}, we have that either $S=1$ or $\sup H=\inf H=\frac{3f_{3}}{2S}$.
It follows from \eqref{2.1-14} that the mean curvature $H$ and the principal curvatures are constant.
Then,  $X: M^{3}\to \mathbb{R}^{4}$
is an isoparametric hypersurface.
$X: M^{3}\to \mathbb{R}^{4}$
is either $\mathbb{S}^{1}(1)\times \mathbb{R}^{2}$, or $\mathbb{S}^{2}(\sqrt{2})\times \mathbb{R}^{1}$ or $\mathbb{S}^{3}(\sqrt{3})$.

\begin{flushright}
$\square$
\end{flushright}

\vskip2mm
\noindent
Acknowledgements. We would like to express our gratitude  to  referees for valuable comments and suggestions.

\end{document}